\documentclass{amsart}
\usepackage{amsmath,amssymb,amsthm, color}
\newcommand {\mbb}{\mathbb}

\newcommand {\ds}{\displaystyle}
\newcommand {\nin}{\noindent}

\newcommand {\seq}{\{\gamma_k\}_{k=0}^{\infty}}
\newcommand {\bs}{\{}
\newcommand {\es}{\}_{k=0}^{\infty}}
\newcommand {\fs}{\sum_{k=0}^{n}}
\newcommand {\lp}{\mathcal{L}-\mathcal{P}}
\newcommand {\be}{\begin{equation}}
\newcommand {\ee}{\end{equation}}
\newcommand {\ab}{{(\alpha,\beta)}}

\newtheorem {thm}{Theorem}
\newtheorem {cor}[thm]{Corollary}
\newtheorem {lm}[thm]{Lemma}

\newtheorem {prop}[thm]{Proposition}
\newtheorem {prob}[thm]{Problem}

\theoremstyle{definition}

\theoremstyle{definition}

\theoremstyle{definition}
\newtheorem {rem}[thm]{Remark}

\theoremstyle{definition}

\begin{document}

\begin{abstract}
Laguerre's theorem regarding the number of non-real zeros of a polynomial and its image under certain linear operators is generalized. This generalization is then used to (1) exhibit a number of previously undiscovered complex zero decreasing sequences for the Jacobi, ultraspherical, Legendre, Chebyshev, and generalized Laguerre polynomial bases and (2) simultaneously generate a basis $B$ and a corresponding $B$-CZDS. An extension to transcendental entire functions in the Laguerre-P\'olya class is given which, in turn, gives a new and short proof of a previously known result due to Piotrowski. The paper concludes with several open questions. 
\end{abstract}

\title{Non-Real Zero Decreasing Operators Related to Orthogonal Polynomials}
\author{Andre Bunton, Nicole Jacobs, Samantha Jenkins, Charles McKenry Jr., Andrzej Piotrowski, and Louis Scott}
\maketitle 

\let\thefootnote\relax\footnote{This research was partially supported by the MAA through an NREUP grant funded by the NSA (grant H98230-13-1-0270) and the NSF (grant DMS-1156582).}

\section{Introduction}\label{intro}
For a non-zero complex valued function $f$ of a complex variable, denote the number (counted according to multiplicity) of real and non-real zeros of $f$ by $Z_R(f)$ and $Z_C(f)$, respectively. For the identically zero function, we define $Z_C(0) = 0=Z_R(0)$. Let $L:\mbb{R}[x]\to\mbb{R}[x]$ be a linear operator.  If $L$ has the property that
\be
Z_C(L(p))\leq Z_C(p)
\ee
for every real polynomial $p$, then $L$ is called a {\it Complex Zero Decreasing Operator}, or CZDO. Such an operator $L$ is diagonal with respect to a basis $B = \{b_k\}_{k=0}^{\infty}$ for $\mbb{R}[x]$ if, and only if, there are real constants $\{\gamma_k\}_{k=0}^{\infty}$ for which 
\be
L(b_k(x)) = \gamma_k b_k(x) \qquad (k=0, 1, 2, \dots).
\ee
In this case, the sequence $\{\gamma_k\}_{k=0}^{\infty}$ is called a {\it Complex Zero Decreasing Sequence for the basis $B$}, or a $B$-CZDS.

A theorem of Laguerre demonstrates the existence of CZDS for the standard basis. We give two versions of his theorem here.
\begin{thm}\label{laguerre}\emph{(Laguerre's Theorem \cite[p. 6]{O}, \cite[p. 23]{CCsurvey})} Let $ p(x) = \fs a_k x^k$ be an arbitrary real polynomial of degree $n$. If $\alpha$ lies outside the interval $(-n,0)$, then 
$$
Z_C\left(\fs (k+\alpha)a_k x^k\right)\leq Z_C\left(\fs a_k x^k\right).
$$
In particular, if $\alpha\geq 0$, then the sequence $\bs k+\alpha \es$ is a CZDS for the standard basis.
\end{thm}
With notation as in Theorem \ref{laguerre},  
$$
x p'(x) + \alpha p(x)  = \fs (k+\alpha)a_k x^k
$$
and Laguerre's theorem may be restated accordingly. 
\begin{thm}\label{laguerreD}\emph{(Laguerre's Theorem; Differential Operator Version)} Let $ p(x)$ be an arbitrary real polynomial of degree $n$. If $\alpha$ lies outside the interval $(-n,0)$, then 
$$
Z_C\left(xp'(x)+  \alpha p(x)\right)\leq Z_C\left(p(x)\right).
$$
In particular, if $\alpha\geq 0$, then the differential operator $xD+\alpha I$ is a CZDO.
\end{thm}

\begin{rem}\label{D}
The differentiation operator $D$ defined by $D(p) := p'$ is a CZDO. This is included in Laguerre's theorem as the special case $\alpha=0$. Indeed, this choice gives
$$
Z_C(p'(x)) = Z_C(xp'(x)) \leq Z_C(p(x)).
$$
\end{rem}
\nin Alternatively, the fact that $D$ is a CZDO can be proved via Rolle's theorem from elementary calculus (see, for example, \cite[p. 2-3]{O}).

Laguerre's theorem is easily extended by iteration to sequences of the form $\bs h(k)\es$, where $h$ is a real polynomial having only real non-positive zeros. This, in turn, leads to a further extension via Hurwitz' theorem to sequences of the form $\bs\varphi(k)\es$, where $\varphi$ is an entire function which is the uniform limit on compact subsets of $\mbb{C}$ of polynomials having only real non-positive zeros (see, for example, \cite[Theorem 1.4]{CCczds}, \cite[p. 6]{O}, \cite{PLaguerre}). We have opted to state Laguerre's theorem in its simplest form to ease the comparison of this theorem with some of its generalizations demonstrated below. 

In 2007, Piotrowski gave a generalization of Laguerre's theorem to obtain a class of $H$-CZDS, where $H$ denotes the set of Hermite polynomials defined by 
$$
H_n(x) = (-1)^n e^{x^2} \frac{d^n}{dx^n} e^{-x^2} \qquad (n=0, 1, 2, \dots).
$$
\begin{thm}\label{hmslaguerregen}\emph{(\cite[p. 57, Proposition 68]{P})}
Suppose $p(x)$ is an arbitrary real polynomial of degree $n$. If $\alpha,\beta,c,d$ are real numbers such that $\alpha \geq 0$, $\beta\geq 0$, and $\alpha+c n \geq 0$,
then 
$$\
\ds Z_C\big(-\beta p''(x)+ (cx+d)p'(x)+ \alpha p(x)\big) \leq Z_C\big(p(x)\big).
$$
In particular, if $\alpha$, $\beta$, and $c$ are all non-negative, then $-\beta D^2+ (cx+d)D + \alpha I$ is a CZDO.
\end{thm}
Since the Hermite polynomials satisfy the differential equation (see, for example, \cite[p. 188]{R})
$$
n H_n(x) = - \frac{1}{2}H_n''(x)+ xH_n'(x)  \qquad (n=0, 1, 2, \dots)
$$
the previous theorem gives, as a special case, the existence $H$-CZDS which can be interpolated by linear polynomials.
\begin{thm}\emph{(\cite[p. 87, Theorem 101]{P})}
Let $p(x) = \fs a_k H_k(x)$ be an arbitrary real polynomial of degree $n$. If $\alpha$ lies outside the interval $(-n,0)$, then 
$$
Z_C\left(\fs (k+\alpha)a_k H_k(x)\right)\leq Z_C\left(\fs a_k H_k(x)\right).
$$
In particular, if $\alpha\geq 0$, then the sequence $\bs k+\alpha \es$ is an $H$-CZDS.
\end{thm}
While no complete characterization of CZDS is currently known for any basis, the characterization of CZDS which can be interpolated by polynomials has been achieved for both the standard basis and the Hermite basis.
\begin{thm}\label{polyintCZDS}\emph{(Craven-Csordas \cite[p. 13]{CCczds})} Let $h(x)$ be a real polynomial.  Then $\bs h(k) \es$ is a CZDS for the standard basis if and only if either
\begin{enumerate}
\item{$h(0)\neq 0$ and $h(x)$ has only real negative zeros}, or
\item{$h(0)=0$ and $h(x)$ is of the form
\be
h(x) = x(x-1)(x-2)\cdots(x-m+1)\prod_{k=1}^{p} (x-b_k)
\ee}
where $m\geq 1$ and $p\geq 0$ are integers and $b_k < m$ for $k=1,2,3,\dots,p$.  
\end{enumerate}
\end{thm}
The previous theorem remains valid {\it mutatis mutandis} if `CZDS for the standard basis' is replaced by `$H$-CZDS' (See \cite[p. 95, Theorem 111]{P}).

The main results of this paper include a generalization of Laguerre's theorem (Theorem \ref{bigcor}), the demonstration of classes of CZDS for the Jacobi, ultraspherical, Legendre, Chebyshev, and generalized Laguerre polynomial bases (Proposition \ref{JaCZDS}, Theorem \ref{JCZDS}, Corollaries \ref{PCZDS} and \ref{TCZDS}, and Theorem \ref{LCZDS}), a method for simultaneously generating a basis $B$ and a corresponding $B$-CZDS (Section \ref{generate}), and the extension of these results to transcendental entire functions in the Laguerre-P\'olya class (Section \ref{LP}).

\section{A Class of Complex Zero Decreasing Operators}\label{CZDO}

This section contains two theorems which generalize Laguerre's theorem. 

\begin{thm}\label{bigthm}
Let $p$ and $q$ be real polynomials, each with degree at least one, and let $\alpha\geq 0$.  Then 
$$
Z_R(f(x)) \geq Z_R(p(x))+ Z_R(q(x))-1
$$
where
$$
f(x) =  q(x) p'(x)+\alpha q'(x) p(x).
$$
\end{thm}

\begin{proof}
When $\alpha=0$, we have 
$$
Z_R( q(x) p'(x)) = Z_R(q(x))+ Z_R(p'(x)) \geq Z_R(p(x))+ Z_R(q(x))-1,
$$
where the last inequality is a consequence of Rolle's theorem.  

We will now suppose $\alpha>0$ for the remainder of the proof.  Suppose $x_0$ is a zero of $p(x) \cdot q(x)$ and write
$$
p(x) = (x-x_0)^m h_1(x) \qquad (h_1(x_0)\neq 0),
$$
$$
q(x) = (x-x_0)^w h_2(x)  \qquad (h_2(x_0)\neq 0).
$$
Then 
$$
f(x) = (x-x_0)^{m+w-1} h_3(x)
$$
where 
$$
h_3(x_0) = (m+\alpha w) h_1(x_0) h_2(x_0) \neq 0.
$$    
(Note:  here we have used the assumption $\alpha>0$.)
That is to say, if $x_0$ is a zero of $p\cdot q$ of multiplicity $m+w$, then $x_0$ is a zero of $f$ of multiplicity $m+w-1$.  We will now complete the proof by demonstrating that $f$ must vanish between consecutive zeros of $p\cdot q$. Define
$$
g(x) := 	\begin{cases}  
				\big[q(x)\big]^\alpha & \qquad \text{if } q(x)\geq 0\\
				
				-\big[-q(x)\big]^\alpha & \qquad \text{if } q(x) <0, \\
				\end{cases}
$$
so that 
$$
\big|q(x) \big|^{1-\alpha} \frac{d}{dx} \left[ g(x) p(x) \right] = q(x) p'(x)+\alpha q'(x) p(x) \qquad (x\notin\{ z|q(z)= 0\})
$$
Let $x_1<x_2$ be consecutive zeros of $p\cdot q$.  Then they are also consecutive zeros of $g\cdot p$, which is continuous on $[x_1, x_2]$ and differentiable on $(x_1, x_2)$.  By Rolle's theorem, $(g\cdot p)'$, and therefore $q(x) p'(x)+\alpha q'(x) p(x)$, has a zero in the interval $(x_1, x_2)$ and the conclusion of the theorem holds.
\end{proof}
We note that Theorem \ref{bigthm} is best possible in the sense that the conclusion does not necessarily hold for any $\alpha<0$. For example, if $\alpha<0$, $p(x) = x^n(x^2+\alpha)$, and $q(x)=x$, then $f(x) = x^n((\alpha+n+2)x^2+ \alpha(\alpha+n))$. Choosing $n=\max\{m\in\mbb{Z} | m\geq 0 \text{ and }  \alpha+m<0\}$ yields $Z_R(f)=n<n+2 = Z_R(p)+ Z_R(q)-1$.

\begin{thm}\label{bigcor}
Let $p$ and $q$ be real polynomials and $\alpha\geq 0$.  Then 
$$
Z_C(q(x) p'(x)+\alpha q'(x) p(x)) \leq Z_C(p(x))+ Z_C(q(x)).
$$
In particular, if $q$ has only real zeros, then $q(x) D+\alpha q'(x) I $ is a CZDO.
\end{thm}

\begin{proof}
We first note that the result is trivial when the function $ q(x)p'(x)+\alpha q'(x)p(x)$ is identically zero. Furthermore, if either $p$ or $q$ is a non-zero constant function, then the result follows from Rolle's theorem as was noted in Remark 3 above. We may, therefore, assume that $p$ and $q$ each have degree at least one. Suppose 
$$
p(x) = \sum_{k=0}^{n}a_k x^k \qquad \text{and} \qquad q(x) = \sum_{k=0}^{m} b_k x^k.
$$
Then the leading term of 
$$
f(x) = q(x) p'(x)+ \alpha q'(x) p(x) 
$$
is $(n+\alpha m) a_n b_m x^{n+m-1}$, so $f$ has degree $n+m-1$.  Applying Theorem \ref{bigthm}, we have
\begin{eqnarray*}
Z_C(f) 	&=& n+m-1 - Z_R(f) \\ 
				&\leq&  n+m-1 - (Z_R(p) + Z_R(q) - 1) \\ 
				&=& n+m-1 - (n-Z_C(p) + m-Z_C(q) - 1) \\ 
				&=& Z_C(p)+Z_C(q). 
\end{eqnarray*}
Therefore, $Z_C(q(x) p'(x)+\alpha q'(x) p(x) ) \leq Z_C(p(x))+ Z_C(q(x))$.
\end{proof}
Note that part of Laguerre's theorem (Theorem \ref{laguerreD}) is obtained when we set $q(x) = x$ in Theorem \ref{bigcor}. 

\begin{rem}{\label{pqr}}
The two theorems in this section can be extended to any number of constants and functions. For example, using the same techniques as above, one can show that 
$$
Z_C(p q r'+\alpha p' q r + \beta p q' r   ) \leq Z_C(p)+Z_C(q)+Z_C(r), 
$$
where $\alpha$ and $\beta$ are non-negative real numbers and $p$, $q$, and $r$ are polynomials. 
\end{rem}

\section{CZDS for the Jacobi Polynomial Basis}\label{JCZDSsection}
\subsection{The Jacobi Polynomials}
We now apply the results of the previous section to demonstrate the existence of CZDS for the Jacobi polynomial basis. Following Rainville \cite[p. 257]{R}, we define the Jacobi polynomials with parameters $\alpha>-1$ and $\beta>-1$ by  
$$
P_n^{(\alpha, \beta)}(x) = \frac{(-1)^n (1-x)^{-\alpha}(1+x)^{-\beta}}{2^n n!} \frac{d^n}{dx^n} \left[(1-x)^{n+\alpha} (1+x)^{n+\beta}\right].
$$
For each non-negative integer $n$, the Jacobi polynomials satisfy the differential equation \cite[p. 258]{R}
\be \label{Jdiffeq} 
\left((x^2-1) D^2 + [(2+\alpha+\beta)x + \alpha-\beta ]D \right) P_{n}^{(\alpha, \beta)}(x)=n(n+1+\alpha+\beta)P_{n}^{(\alpha, \beta)}(x).
\ee

\begin{prop}\label{JaCZDS}
The sequence $\bs k(k+1+\alpha+\beta) \es$ is a $P^{\ab}$-CZDS. 
\end{prop}
\begin{proof} Define the linear operator $L:\mbb{R}[x]\to\mbb{R}[x]$ by 
$$
L\left(P_k^{\ab}(x)\right) = k(k+1+\alpha+\beta)  P_k^{\ab}(x) \qquad (k=0, 1, 2, \dots)
$$
so that, by linearity, 
$$
L\left(\sum_{k=0}^{n} a_k P_k^{\ab}(x)\right) = \sum_{k=0}^{n} a_k L(P_k^{\ab}(x)) = \sum_{k=0}^{n} a_k k(k+1+\alpha+\beta) P_k^{\ab}(x).
$$
Our goal, then, is to show that $L$ is a CZDO. From the differential equation (\ref{Jdiffeq}), the linear operator $L$ is equal to the differential operator
$$
L  = \left((x^2-1)D+[(2+\alpha+\beta)x + \alpha-\beta ]I\right)D.
$$
If, in Remark \ref{pqr}, we take $p(x)=x-1$, $q(x) = x+1$ and replace $\alpha$ and $\beta$ by $\alpha + 1$ and $\beta+1$, respectively, then we see that 
$$
(x^2-1)D+[(2+\alpha+\beta)x + \alpha-\beta ]I \qquad (\alpha, \beta >-1)
$$ 
is a complex zero decreasing operator. Thus, $L$ is the composition of two CZDO (recall that D is a CZDO as discussed in Remark 3 above) and so it is a CZDO itself.
\end{proof} 
\subsection{Operator Identities}\label{opid}
In order to extend the preceding result, we will develop a number of operator identities. We consider two operators $L_1$ and $L_2$ on $\mbb{R}[x]$ to be equal if $L_1(p) = L_2(p)$ for every real polynomial $p$. For example, as a consequence of the product rule for differentiation, $(Dx)p(x) = xp'(x)+ p(x)$, and thus we obtain the equality
\be\label{Dx}
Dx = xD+I.
\ee
\begin{prop}\label{genD'd} 
Suppose $\{g_k(x)\}_{k=0}^{m}$ is a sequence of polynomials satisfying $\text{deg}(g_k)\leq k$ for all $k$. Then 
$$
D^n \sum_{k=0}^{m} g_k(x) D^k = \left[\sum_{j=0}^{m} \sum_{k=j}^{m} \binom{n}{k-j} g_{k}^{(k-j)}(x) D^{j}\right] D^n.
$$
\end{prop}

\begin{proof}
We first note that we are following the convention that $\ds \binom{n}{k}=0$ whenever $k>n$. 

Using the fact that the derivative operator is linear, applying Leibniz' formula for the $n$th derivative of a product, and noting our assumption on the degree of the polynomials $g_k$, we have
\begin{eqnarray*}
 D^n \sum_{k=0}^{m} g_k(x) D^k &=& \sum_{k=0}^{m} \sum_{i=0}^{k} \binom{n}{i}g_k^{(i)}(x) D^{k+n-i}\\
&=& \left[\sum_{k=0}^{m} \sum_{i=0}^{k} \binom{n}{i}g_k^{(i)}(x) D^{k-i}\right]D^n.
\end{eqnarray*}
Making the substitution $j=k-i$ and then switching the order of summation gives
\begin{eqnarray*}
 D^n \sum_{k=0}^{m} g_k(x) D^k &=& \left[\sum_{k=0}^{m} \sum_{j=0}^{k} \binom{n}{k-j}g_k^{(k-j)}(x) D^{j}\right]D^n\\
 &=& \left[\sum_{j=0}^{m} \sum_{k=j}^{m} \binom{n}{k-j}g_k^{(k-j)}(x) D^{j}\right]D^n,
\end{eqnarray*}
as desired.
\end{proof}
In what follows, we will make frequent use of Proposition \ref{genD'd} with $m=2$, which asserts that if
\begin{equation} \label{gD1}
L := D^n (g_2(x) D^2 + g_1(x) D + g_0(x) I),
\end{equation}
then
\begin{equation} \label{gD2}
L = \left(g_2(x) D^2 + (n g_2'(x)+ g_1(x))D + \left(\binom{n}{2} g_2''(x) + n g_1'(x) + g_0(x)\right)I\right)D^n,
\end{equation}
provided deg$(g_k)\leq k$ for all $k$.
\subsection{Ultraspherical Polynomials} \label{JlambdaCZDS}
We now focus on the Jacobi polynomials for which $\alpha=\lambda=\beta$, which are called the ultraspherical polynomials (see, e.g., \cite[p.143]{R}). To ease notation, we define
$$
P_n^{(\lambda)}(x) := P_n^{(\lambda, \lambda)}(x) \qquad (\lambda>-1; n=0, 1, 2, \dots). 
$$
With this choice, the differential equation (\ref{Jdiffeq}) takes on the form 
\be \label{lambdade}
\left[(x^2-1) D^2 + (1+\lambda)2x D\right] P_n^{(\lambda)}(x) = n(n+1+2\lambda)  P_n^{(\lambda)}(x).
\ee 
Due to the frequent use of the operator involved in the previous equation we define, for any $a\in\mbb{R}$, 
\be \label{Phi}
\Phi_a := (x^2-1)D + (1+a)2xI. 
\ee
\begin{lm} \label{JnD'd}
Suppose $c\in \mbb{R}$ and $\lambda>-1$. Then, for all non-negative integers $n$, 
$$
D^n (\Phi_{\lambda}D -n(n+1+2\lambda)I) = \big( \Phi_{\lambda+n} \big)D^{n+1},
$$
where $\Phi_a$ is defined in equation \emph{(\ref{Phi})}.
\end{lm}

\begin{proof}
This is an immediate application of equations (\ref{gD1}) and (\ref{gD2}).
\end{proof}

We now use a product notation for composition of operators. Since differential operators need not commute, care is required in using this notation. For a collection of operators $L_1, L_2, \dots, L_n$ on $\mbb{R}[x]$, we define
$$
\left(\prod_{k=1}^{n} L_k\right) p := (L_1 L_2 \cdots L_n)p = L_1(L_2(\cdots (L_n(p))))\qquad (p\in\mbb{R}[x]).
$$
\begin{prop}\label{mud}
Let $w$ be a positive integer and $\{m_k\}_{k=0}^{w-1}\subset\mbb{N}$. Then
$$
\prod_{k=0}^{w-1} (\Phi_{\lambda}D-k(k+1+2\lambda)I)^{m_k} = \left[\prod_{k=0}^{w-1}  \left[( \Phi_{\lambda+k}D)^{m_k-1} \Phi_{\lambda+k} \right]\right]D^{w},
$$
where $\Phi_a$ is defined by equation \emph{(\ref{Phi})}.
\end{prop}

\begin{proof}
We will argue by mathematical induction. The case $w=1$ is clear. Now suppose that the result is true for some integer $w\geq 1$ and fix natural numbers $m_0, m_1, \dots, m_w$. Then 
\be\label{theta1}
\prod_{k=0}^{w} (\Phi_{\lambda} D -k(k+1+2\lambda)I)^{m_k} = \Theta D^{w} (\Phi_{\lambda}D-w(w+1+2\lambda)I)^{m_w},
\ee
where
\be\label{theta2}
\Theta = \prod_{k=0}^{w-1}  \left[\big( \Phi_{\lambda+k}D\big)^{m_k-1}\big(\Phi_{\lambda+k}\big)\right].
\ee
Applying lemma \ref{JnD'd} a total of $m_w$ times we see that 
\be\label{theta3}
D^{w} ( \Phi_{\lambda}D-w(w+1+2\lambda)I)^{m_w}  = ( \Phi_{\lambda+w}D )^{m_w}D^{w}.
\ee
Together, equations (\ref{theta1}), (\ref{theta2}), and (\ref{theta3}) show that 
$$
\prod_{k=0}^{w} (\Phi_{\lambda}D-k(k+1+2\lambda)I)^{m_k} = \left[\prod_{k=0}^{w}  \left[\big( \Phi_{\lambda+k}D \big)^{m_k-1}\big(\Phi_{\lambda+k}\big)\right] \right]D^{w+1}
$$
as desired. 
\end{proof}

We are now in a position to demonstrate the existence of several $P^{(\lambda)}$-CZDS for any fixed $\lambda>-1$.

\begin{thm}\label{JCZDS} If $\lambda>-1$,  $w$ is a positive integer, and $\{m_k\}_{k=0}^{w-1}\subset\mbb{N}$, then the sequence
\be\label{Jfallingfactorial}
\left\{ \prod_{k=0}^{w-1}(n(n+1+2\lambda)-k(k+1+2\lambda))^{m_k} \right\}_{n=0}^{\infty}
\ee
is a $P^{(\lambda)}$-CZDS, where $P^{(\lambda)}$ is the set of ultraspherical polynomials.
\end{thm}

\begin{proof}
Let the linear operator $L:\mbb{R}[x]\to\mbb{R}[x]$ be defined by 
$$
L(P_n^{(\lambda)}(x)) = \left(\prod_{k=0}^{w-1}(n(n+1+2\lambda)-k(k+1+2\lambda))^{m_k} \right)P_n^{(\lambda)}(x).
$$
From the differential equation (\ref{lambdade}), we have
$$
L = \prod_{k=0}^{w-1} ((x^2-1)D^2+(1+\lambda)2xD-k(k+1+2\lambda)I)^{m_k},
$$
or, using the notation in equation (\ref{Phi}) and applying Proposition \ref{mud},
$$
L =  \prod_{k=0}^{w-1} ( \Phi_{\lambda} D-k(k+1+2\lambda)I)^{m_k} = \left[\prod_{k=0}^{w-1}  \left[( \Phi_{\lambda+k}D)^{m_k-1} \Phi_{\lambda+k} \right]\right]D^{w}.
$$
The operator $L$ is, therefore, a composition of individual operators, each of which are CZDO. This can be seen by appealing to Theorem \ref{bigcor}, which shows that $\Phi_a$ is a CZDO whenever $a>-1$. 
\end{proof}

\subsection{CZDS for Legendre Basis}

The polynomials 
$$
P_n(x) := P_n^{(0)}(x) = P_n^{(0,0)}(x) \qquad (n=0, 1, 2, \dots)
$$
are known as the Legendre polynomials (see \cite[p. 254]{R}).

In \cite{BDFU}, open question (4) conjectures that a certain type of falling factorial sequence is a multiplier sequence for the Legendre basis ($P$-MS). Since every $P$-CZDS is a $P$-MS, we can apply the results of the previous section to settle a variation of this question. 

\begin{cor}\label{PCZDS} If $w$ is a positive integer and $\{m_k\}_{k=0}^{w-1}\subset\mbb{N}$, then the sequence
\be\label{Pfallingfactorial}
\left\{ \prod_{k=0}^{w-1}(n(n+1)-k(k+1))^{m_k} \right\}_{n=0}^{\infty} = \left\{ \prod_{k=0}^{w-1}((n+k+1)(n-k))^{m_k} \right\}_{n=0}^{\infty}
\ee
is a CZDS for the Legendre basis.
\end{cor}

\begin{proof}
Apply Theorem \ref{JCZDS} with $\lambda=0$.
\end{proof}

Corollary \ref{PCZDS} strengthens and extends some of the results obtained in \cite{BDFU} by showing that $\bs k^2+k \es$ is a $P$-CZDS, and by demonstrating the existence of $P$-CZDS (and hence $P$-multiplier sequences) which are not products of quadratic $P$-multiplier sequences.

\subsection{CZDS for the Chebyshev Basis}
The Chebyshev polynomials $\mathcal{T} = \{T_n(x)\}$ and $\mathcal{U} = \{U_n(x)\}$ of the first and second kind, respectively, can be defined by (see \cite[p. 301]{R})
$$
T_n(x) := \frac{n!}{(\frac{1}{2})_n} P_n^{(-1/2)}(x) \qquad (n=0, 1, 2, \dots),
$$
$$
U_n(x) := \frac{(n+1)!}{(\frac{3}{2})_n} P_n^{(1/2)}(x) \qquad  (n=0, 1, 2, \dots),
$$
where $(a)_n := a(a+1)\cdots (a+n-1) $ is the rising factorial. In \cite[Lemma 156]{P} it is shown that a sequence $\{\gamma_k\}_{k=0}^{\infty}$ is a CZDS for a simple set $Q = \{q_k(x)\}_{k=0}^{\infty}$ if and only if it is a $\widehat{Q}$-CZDS, where $\widehat{Q}$ consists of the polynomials
$$
\widehat{q}_n(x) = c_n q_n(\alpha x + \beta) \qquad (\beta\in \mbb{R}; \alpha, c_n \in \mbb{R}\setminus\{0\}).
$$
Combining this with Theorem \ref{JCZDS}, we arrive at the following corollary. 
\begin{cor} \label{TCZDS}
If $w$ is a positive integer and $\{m_k\}_{k=0}^{w-1}\subset\mbb{N}$, then 
\begin{enumerate}
\item the sequence $\left\{ \prod_{k=0}^{w-1}(n^2-k^2)^{m_k} \right\}_{n=0}^{\infty}$ is a $\mathcal{T}$-CZDS, and
\item the sequence $\left\{ \prod_{k=0}^{w-1}(n(n+2)-k(k+2))^{m_k} \right\}_{n=0}^{\infty}$ is a $\mathcal{U}$-CZDS.
\end{enumerate}
\end{cor}

\begin{proof}
Apply Theorem \ref{JCZDS} with $\lambda=-1/2$ and again with $\lambda = 1/2$.
\end{proof}

\section{Simultaneous Generation of a Basis $B$ and a class of $B$-CZDS}\label{generate}

Given a basis $B$ and a sequence $\seq$, a typical strategy in showing that $\seq$ is a $B$-CZDS is to find a differential operator representation for the diagonal operator which is a CZDO. In this section, we begin with a known CZDO and use it to simultaneously generate a basis $B$ and a corresponding $B$-CZDS. Our results focus on bases which are {\it simple sets}, i.e., for which deg$(b_k) = k$ for all $k$.
\begin{thm}
Let $\alpha\geq 0$ and let 
$$
q(x) = c_0+ c_1x + \cdots + c_r x^r \qquad (r \geq 1, c_r \neq 0)
$$ 
be a real polynomial with only real zeros. If the simple set of real polynomials $B = \{b_k(x)\}_{k=0}^{\infty}$ satisfy the differential equation
\begin{equation}\label{bde}
\gamma_n b_n(x) =  q(x) b_n^{(r)}(x)+\alpha q'(x) b_n^{(r-1)} (x) \qquad (n=0, 1, 2, \dots),
\end{equation}
where $\{\gamma_n\}_{n=0}^{\infty}$ is a sequence of real numbers, then 
$$
\gamma_n = c_r \frac{n!(n+(\alpha-1)r+1)}{(n-r+1)!} \qquad (n=0, 1, 2, \dots)
$$
and $\{\gamma_n\}_{n=0}^{\infty}$ is a $B$-CZDS.
\end{thm}

\begin{proof}
Writing the polynomials $b_n$ in the form
$$
b_n(x) = \sum_{k=0}^{n} \frac{a_{n,k}}{k!} x^k \qquad (a_{n,n}\neq 0), 
$$
we have
$$
b_n^{(j)}(x) = \sum_{k=0}^{n-j}\frac{a_{n,k+j}}{k!}{x^{k}}
$$
and so the leading coefficient of the $j^{th}$ derivative of $b_n$ is $a_{n,n} / (n-j)!$. 
\vskip .1 in 
\nin {\bf Case 1:} $n < r-1$.  In this case, the differential equation becomes
$$
\gamma_n b_n(x) = 0
$$
from which we must have $\gamma_n=0$.
\vskip .1 in 
\nin {\bf Case 2:} $n=r-1$. In this case, the differential equation becomes
$$
\gamma_{r-1} b_{r-1}(x) = \alpha q'(x) {a_{r-1,r-1}} 
$$
and equating leading coefficients, we obtain
$$
\gamma_{r-1} \frac{a_{r-1,r-1}}{(r-1)!} = \alpha r c_r {a_{r-1,r-1}} 
$$
or
$$
\gamma_{r-1}  = \alpha r! c_r.  
$$
\vskip .1 in 
\nin {\bf Case 3:} $n > r-1$.
If we assume (\ref{bde}) is satisfied, then we can equate leading coefficients on each side of the equals sign to obtain
$$
\gamma_n \frac{a_{n,n}}{n!} =  c_r \frac{a_{n,n}}{(n-r)!}+ \alpha r c_r \frac{a_{n,n}}{(n-(r-1))!} .
$$ 
From this, it follows that 
$$
\gamma_n = c_r \frac{n!(n+(\alpha-1)r+1)}{(n-r+1)!}
$$
as desired. 

To show that $\{\gamma_n\}_{n=0}^{\infty}$ is a $B$-CZDS, suppose 
$$
g(x) = \sum_{k=0}^{m} d_k b_k(x) \qquad (d_m \neq 0)
$$
is a real polynomial. Then 
$$
Z_C(g(x)) \geq Z_C(g^{(r-1)}(x)) \geq Z_C( q(x) g^{(r)}(x)+\alpha q'(x) g^{(r-1)}(x)),
$$
where we have made use of Remark \ref{D} and Theorem \ref{bigcor}. 
Since
\begin{eqnarray*}
 q(x) g^{(r)}(x)+\alpha q'(x) g^{(r-1)}(x) &=& \sum_{k=0}^m d_k (q(x) b_k^{(r)}(x)+\alpha q'(x) b_k^{(r-1)}(x) ) \\
&=&  \sum_{k=0}^m \gamma_k d_k b_k(x),
\end{eqnarray*}
the desired result is obtained.
\end{proof}

As an example, if we choose $q(x) = (x+1)^3$ and $\alpha = 1$, then the corresponding sequence would be $\gamma_n = (n+1)n(n-1)$, and we would need to find a simple set $B = \{b_n(x)\}_{n=0}^{\infty}$ which solves the differential equation
\begin{equation}{\label{DEex}}
(n+1) n (n-1) b_n(x) = (x+1)^3 b_n'''(x)+3(x+1)^2 b_n''(x)  \qquad (n=0, 1, 2, \dots).
\end{equation}
With some effort, one finds that sets $B$ which solve equation (\ref{DEex}) do exist, each of which has the form 
$$
b_0(x) = r,
$$$$
b_1(x) = sx+t,
$$$$
 b_n(x) = c_n (x+1)^n \quad (n=2, 3, 4, \dots)
$$
where $t\in\mbb{R}$ and $r, s, c_2, c_3, \dots$ are any (fixed) non-zero real numbers. Thus, the sequence 
$$
\{(n+1)n(n-1)\}_{n=0}^{\infty}
$$ 
is a $B$-CZDS for any such basis $B$.

\section{An Extension to Certain Transcendental Entire Functions}

\subsection{The Laguerre-P\'olya Class}\label{LP}
A real entire function $ \varphi$ is said to belong to the \textit{Laguerre-P\'olya class}, denoted $\varphi\in \mathcal{L-P}$, if it can be written in the form
\be
\varphi(x)=cx^me^{-ax^2+bx}\prod_{k=1}^\omega \left(1+\frac{x}{x_k}\right)e^{-x/x_k}
\ee
where $b,c,x_k\in \mbb{R}$,  $m$ is a non-negative integer, $a\geq 0$, $0\leq \omega \leq \infty$, and $ \sum_{k=1}^{\omega} {x_k^{-2}}<\infty$.  

An alternate characterization of this class is as follows: $\varphi\in\lp$ if and only if $\varphi$ is the uniform limit on compact subsets of $\mbb{C}$ of real polynomials having only real zeros (See, for example, \cite[Ch. VIII ]{Levin} or \cite[Satz 9.2]{O}). This point of view, together with Hurwitz' theorem (see \cite[p. 4]{Marden}) allow us to obtain some useful extensions of results in Section \ref{CZDO}.

\begin{thm}\label{LPcor}
Suppose $\varphi$ belongs to the class $\lp$, $p$ and $q$ are real polynomials, and $\alpha\geq 0$.  Then 
$$
Z_C(\varphi  q p'+\alpha (\varphi q )' p ) \leq  Z_C(p)+Z_C(q).
$$
\end{thm}

\begin{proof}
Suppose $\bs f_k \es$ is a sequence of real polynomials with only real zeros which converge uniformly on compact subsets of $\mbb{C}$ to $\varphi.$ By Theorem \ref{bigcor},
$$
Z_C( f_k q p'+\alpha (f_k q)' p ) \leq Z_C(p)+Z_C(q) \qquad (k=0, 1, 2, \dots).
$$
Since $ f_k q p'+\alpha (f_k q)' p $ converges uniformly on compact subsets of $\mbb{C}$ to $\alpha (\varphi q)' p + \varphi q p'$, Hurwitz' theorem gives the desired result. 
\end{proof}
In order to prove an extension of Laguerre's theorem related to $H$-CZDS (Theorem \ref{hmslaguerregen} in the Introduction), Piotrowski first proved a special case as a lemma. We now show how to obtain a new proof of this lemma using Theorem \ref{LPcor}.

\begin{cor}\label{lmgenBC}\emph{(\cite[p. 55, Lemma 67]{P})}
Suppose that $p(x)$ is a real polynomial of degree $n$.  If $c,d,\beta$ are real numbers such that $c\geq 0$ and $\beta\geq 0$, then
$$
Z_C\big( (cx+d)p(x)-\beta p'(x)\big) \leq Z_C\big(p(x)\big).
$$
\end{cor}

\begin{proof}
If $\beta = 0$, the result clearly holds. If $\beta>0$ we may appeal to corollary \ref{LPcor} with $\alpha = \beta^{-1}$, $q(x) = 1$, and 
$$
\varphi(x) = -\exp\left(-\frac{c}{2}x^2- d\,x\right) \qquad (c\geq 0, d\in\mbb{R})
$$
to obtain the desired result.
\end{proof}

\subsection{CZDS for the Generalized Laguerre Polynomial Basis}

In this section, we combine the results of the previous section with the methods of Section \ref{JlambdaCZDS} to obtain a class of CZDS for the generalized Laguerre polynomial basis, defined by 
$$
L_n^{(\alpha)}(x) :=\sum_{k=0}^{n} \binom{n+\alpha}{n-k} \frac{(-x)^k}{k!} \qquad (\alpha>-1; n=0, 1, 2, \dots).
$$
The generalized Laguerre polynomials satisfy the differential equation (see, e.g., \cite[p. 204]{R})
\begin{equation}\label{Ldiffeq}
 - x \frac{d^2}{dx^2} L_n^{(\alpha)}(x) +(x-(\alpha+1)) \frac{d}{dx}L_n^{(\alpha)}(x) = n L_n^{(\alpha)}(x).
\end{equation}
Just as with the Jacobi basis, we will develop a number of operator identities in order to arrive at a collection of $L^{(\alpha)}$-CZDS. We begin by defining, for any $a\in\mbb{R}$, 
\be\label{Psi}
\Psi_a = -xD + (x-(a+1))I
\ee
\begin{lm} \label{LnD'd}
Suppose $c,\alpha\in \mbb{R}$.  Then, for all non-negative integers $n$, 
$$
D^n ( \Psi_{\alpha} D - n I) = \Psi_{\alpha+n} D^{n+1}.
$$
\end{lm}

\begin{proof}
This is an immediate application of equations (\ref{gD1}) and (\ref{gD2}) which appear after the proof of Proposition \ref{genD'd} in Section \ref{opid}.
\end{proof}

\begin{prop}\label{Lmud}
Let $w$ be a positive integer and $\{m_k\}_{k=0}^{w-1}\subset\mbb{N}$. 
Then
$$
\prod_{k=0}^{w-1} (\Psi_{\alpha}D -kI)^{m_k} = \left[\prod_{k=0}^{w-1}  \left[(\Psi_{\alpha+k}D)^{m_k-1} \Psi_{\alpha+k}  \right]\right]D^{w},
$$
where $\Psi_a$ is defined in equation (\ref{Psi}).
\end{prop}

\begin{proof}
We will argue by mathematical induction. The case $w=1$ is clear. Now suppose that the result is true for some integer $w\geq 1$ and fix natural numbers $m_0, m_1, \dots, m_w$. Then 
\be\label{Ltheta1}
\prod_{k=0}^{w} (\Psi_{\alpha}D -kI)^{m_k} = \prod_{k=0}^{w-1}  \left[\big(\Psi_{\alpha+k}D \big)^{m_k-1} \Psi_{\alpha+k} \right] D^{w} ( \Psi_{\alpha}D -w I)^{m_w}.
\ee
Applying lemma \ref{LnD'd} a total of $m_w$ times we see that 
\be\label{Ltheta3}
D^{w} ( \Psi_{\alpha}D -wI)^{m_w}  = (\Psi_{\alpha+w}D)^{m_w}D^{w} =  (\Psi_{\alpha+w}D)^{m_w-1} \Psi_{\alpha+w} D^{w+1}.
\ee
Together, equations (\ref{Ltheta1}) and (\ref{Ltheta3}) give the desired result.
\end{proof}

In order to use the operator identities above to find a collection of $L^{(\alpha)}$-CZDS for any $\alpha>-1$, we will use the result of Section \ref{LP}.

\begin{lm}\label{Llem}
For any $a>-1$, the operator $ \Psi_a = - x D+(x-(a+1))I$ is a CZDO.
\end{lm}

\begin{proof}
Suppose $a>-1$ and set $c=a+1$. By Theorem \ref{LPcor}, for any real polynomial $p$, 
$$
Z_C\left( c \frac{d}{dx}\left( -x\exp(-x/c) \right)p(x) + (-x \exp(-x/c))p(x)   \right) \leq Z_C(p(x)).
$$
The smaller quantity above simplifies to
$$
Z_C\left( (- x p'(x)+(x-c)p(x) )\exp(-x/c)  \right).
$$
Since the exponential function never vanishes, we have shown that
$$
Z_C(\Psi_a p(x)) = Z_C( - x p'(x)+(x-c)p(x))\leq Z_C(p(x)).
$$
\end{proof}

We now arrive at the main theorem of this section.

\begin{thm}\label{LCZDS} Fix $\alpha>-1$. If $w$ is a positive integer and $\{m_k\}_{k=0}^{w-1}\subset\mbb{N}$, then the sequence
\be\label{Lfallingfactorial}
\left\{ \prod_{k=0}^{w-1}(n-k)^{m_k} \right\}_{n=0}^{\infty} 
\ee
is an $L^{(\alpha)}$-CZDS.
\end{thm}

\begin{proof}
Let the linear operator $\Theta:\mbb{R}[x]\to\mbb{R}[x]$ be defined by 
$$
\Theta(L_n^{(\alpha)}(x)) := \left( \prod_{k=0}^{w-1}(n-k)^{m_k}  \right)L_n^{(\alpha)}(x).
$$
Combining the differential equation (\ref{Ldiffeq}), the notation defined in equation (\ref{Psi}), and Proposition \ref{Lmud}, we have
$$
\Theta = \prod_{k=0}^{w-1} (\Psi_{\alpha}D - k I )^{m_k} = \left[\prod_{k=0}^{w-1}  \left[(\Psi_{\alpha+k}D)^{m_k-1} \Psi_{\alpha+b}\right]\right]D^{w}.
$$
The operator $\Theta$ is, therefore, a composition of individual operators, each of which are CZDO. This can be seen by appealing to Lemma \ref{Llem}.
\end{proof}

Theorem \ref{LCZDS} is a significant generalization and extension of a theorem due to Forg\'acs and Piotrowski \cite[Theorem 4.4]{FP} and a stronger result on a narrower class of sequences than those characterized by Br\"and\'en and Ottergren in \cite{BO}.

\section{Open Questions}
Any sequence of the form 
$$
\{k(k-1)\cdots (k-(m-1))\}_{k=0}^{\infty}
$$
(the ``falling-factorial sequence'') is a CZDS for the standard basis. By Corollary \ref{TCZDS}, any sequence of the form 
$$
\{k^2(k^2-1)\cdots(k^2-(m-1)^2)\}_{k=0}^{\infty}
$$
is a $T$-CZDS. The similarity of these results lead us to wonder if an analog of Theorem \ref{polyintCZDS} could be obtained for the Chebyshev basis. 
\begin{prob} \label{P1}
Find a complete characterization of polynomials $h$ for which $\bs h(k)\es$ is a $T$-CZDS, where $T$ denotes the Chebyshev basis.
\end{prob}
We note that the characterization will be different from that of the standard basis, since the sequence $\{k\}_{k=0}^{\infty}$ is not a $T$-CZDS. 

The results on ultraspherical and Laguerre CZDS also have a falling factorial nature which leads us to consider the more general problem.
\begin{prob} \label{P2}
For any basis $B$, find a complete characterization of polynomials $h$ for which $\bs h(k)\es$ is a $B$-CZDS.
\end{prob}
Recall that this problem has been solved when the basis is taken to be either the standard basis or the Hermite basis. The result \cite[Lemma 157]{P} solves the problem for any affine transformation of the standard basis or the Hermite basis. To date, Problem \ref{P2} remains unsolved for any other choice of the basis $B$.

As it was mentioned earlier, no complete characterization of CZDS for the standard basis is known. In particular, it is not known whether or not every rapidly decreasing sequence (such as $\{\exp(-k^3)\}_{k=0}^{\infty}$) is a CZDS for the standard basis (see \cite[Problem 4.8]{CCsurvey} for more details). A theorem of Piotrowski gives a connection between these and CZDS for other bases.
\begin{thm}\label{thmqczdsczds}\cite[Theorem 159]{P}
Let $B = \{q_k(x)\}_{k=0}^{\infty}$ be a simple set of polynomials.  If the sequence $\{\gamma_k\}_{k=0}^{\infty}$ is a $B$-CZDS, then the sequence $\{\gamma_k\}_{k=0}^{\infty}$ is a CZDS for the standard basis.
\end{thm}
This prompts us to state a weaker version of Problem 4.8(a) of \cite{CCsurvey} which may be easier to settle.
\begin{prob}\label{P3}
Is there a simple set $B$ for which $\{\exp(-k^3)\}_{k=0}^{\infty}$ is a $B$-CZDS?
\end{prob}
We mention that our methods of simultaneously generating a basis and CZDS may apply. However, the original operator will have to be modified as all of our methods generated sequences which can be interpolated by polynomials. 

Our last open problem has to do with the existence of solutions of differential equations. We believe the following problem can be answered in the affirmative and this result would give a more complete picture of the main result in section \ref{generate}.
\begin{prob}\label{P4}
Let $\alpha\geq 0$ and let $q$ be a real polynomial of degree $r$ with leading coefficient $c_r$. 
Is there a simple set of polynomials $B = \{b_k(x)\}_{k=0}^{\infty}$ that satisfy the differential equation
\begin{equation}\label{P4bde}
 c_r \frac{n!(n+(\alpha-1)r+1)}{(n-r+1)!} b_n(x) =  q(x) b_n^{(r)}(x)+ \alpha q'(x) b_n^{(r-1)} (x)
\end{equation}
for $n=0, 1, 2, \dots$?
\end{prob}

\section{Acknowledgment}
The authors would like to thank the MAA, NSA, and NSF for their financial support of this project and the mathematics program faculty and staff at UAS for their moral and administrative support. The fifth author would also like to recognize Dr. George Csordas and Dr. Tam\'as Forg\'acs for their inspiration, encouragement, and helpful suggestions.

\end{document}